\documentclass[11pt,reqno,oneside]{amsart}
\usepackage[letterpaper]{geometry}
\usepackage[rgb,table]{xcolor}
\usepackage[colorlinks, allcolors=blue]{hyperref}
\usepackage{amsmath,amsthm,amssymb,tikz,ifthen,enumerate,xhfill}

\allowdisplaybreaks % align/align* environment can break over pages

\usepackage{cleveref} % Use \Cref to include "Theorem ", etc.
\newtheorem{thm}{Theorem}

\newtheorem{lem}[thm]{Lemma}

\newtheorem{conjecture}{Conjecture}
\theoremstyle{definition}

\newtheorem{remark}[thm]{Remark}
\newtheorem*{remark*}{Remark}
\newcommand\<{\begin{equation}} \renewcommand\>{\end{equation}}

\usepackage[labelfont=sc, justification=centering, textformat=period]{caption}
\makeatletter \g@addto@macro\@floatboxreset\centering \makeatother % to center the figures themselves

\let\tilde\widetilde

\renewcommand\bar[1]{\,\overline{\!#1\!}\,}
\renewcommand\dots{...} \renewcommand\ldots{...}

\newcommand\N{{\mathbb N}}
\newcommand\R{{\mathbb R}}

\newcommand\Z{{\mathbb Z}}

\newcommand\Params{{\mathcal P}}
\newcommand\Square{{\mathcal S}}
\newcommand\drawP[3][black]{ 
	\fill [#2, domain=-#3:-1/#3,samples=25] plot (\x, -1/\x) -- (0,#3) -- (0,1) -- (-1,0) -- (-#3,0) -- cycle;
	\draw [#1, domain=-#3:-1/#3,samples=25] plot (\x, -1/\x);
	\draw [#1] (-#3,0) -- (-1,0) -- (0,1) -- (0,#3);
}
\renewenvironment{cases}{  \left\{\!\!\begin{array}{ll} }{ \end{array}\!\!\right. }
\def\comp{k}
\def\half{{\tfrac12}}
\def\third{{\tfrac13}}
\colorlet{artincolor}{red}
\colorlet{hurwitzcolor}{green!67!black}

\newcounter{commentcounter}\newcommand\COMMENT[2][red]{\stepcounter{commentcounter}\rlap{\smash{$^{\fcolorbox{#1}{#1!15}{\scriptsize\ifthenelse{\equal{#1}{green}}{\color{green!67!black}}{\color{#1}}\!\!{\bf\thecommentcounter}\!\!}}$}}\marginpar{\!\!\parbox{2.8cm}{\raggedright\small \ifthenelse{\equal{#1}{green}}{\color{green!67!black}}{\color{#1}} \textbf{\thecommentcounter.}\,#2}}} \usepackage{silence} 
\title{On the topological entropy of  (\textit{\lowercase{a}},\textit{\lowercase{b}})-continued fraction transformations}
\author{Adam Abrams}
\address{Faculty of Pure and Applied Mathematics, Wroc\l{}aw University of Science and Technology, Wroc\l{}aw, 50370, Poland}
\email{the.adam.abrams@gmail.com}
\author{Svetlana Katok}
\address{Department of Mathematics, The Pennsylvania State University, University Park, PA 16802, USA}
\email{sxk37@psu.edu}
\author{Ilie Ugarcovici}
\address{Department of Mathematical Sciences, DePaul University, Chicago, IL 60614, USA}
\email{iugarcov@depaul.edu}

\thanks{The second author was partially supported by NSF grant DMS 1602409.}
\date{\today}

\begin{document}

\begin{abstract}
We study the topological entropy of a two-parameter family of maps related to $(a,b)$-continued fraction algorithms and prove that it is constant on a square within the parameter space (two vertices of this square correspond to well-studied continued fraction algorithms). The proof uses conjugation to maps of constant slope.
We also present experimental evidence that the topological entropy is flexible (i.e., takes any value in a range) on the whole parameter space. 
\end{abstract}

\maketitle

\section{Introduction}
The dynamics of piecewise monotone interval maps, and in particular their topological entropy and conjugation to maps of constant slope, has been a rich area of investigation, going back to fundamental work of Parry~\cite{Parry66}. See also the monographs~\cite{ALM,dMvS} and references therein. Within this class of interval maps, a considerable amount of work has been done for unimodal maps. Boundary maps associated to co-compact Fuchsian groups (see~\cite{AKU-Rigidity}) provide an important family of piecewise monotone examples with multiple discontinuity points. In this paper, we study the topological entropy of a two-parameter family of boundary maps $f_{a,b}:\bar\R \to \bar\R$, where $\bar\R = \R \cup \{\infty\}$, associated to the modular group $\mathrm{PSL}(2,\mathbb Z)$. These transformations were introduced in~\cite{KU-JMD} and are given by
\< \label{eqn:fab}
	f_{a,b}(x) := \begin{cases}
		x+1  &\text{ if } x< a \\
		-\dfrac1{x} &\text{ if } a\le x<b \\
		x-1  &\text{ if } x\ge b,
	\end{cases}
\>
where the parameters $a,b$ belong to the set
\[
	\Params := \{(a,b)\in \R^2 \,|\, a\leq 0\leq b,\,b-a\geq 1,\,-ab\leq 1\}. % FULL SPACE
	%\Params := \{ (a,b) \in \R^2 \,| -1 \le a \le 0 \le b \le 1,\; b-a \ge 1 \}. % TRIANGLE ONLY
\]

The maps $f_{a,b}$ can be used to construct continued fraction expansions: 
for any $x\in\R$, 
\[
x=n_0-\cfrac1{n_1-\cfrac1{n_2-\cfrac1{n_3-\cdots}}}:=\lfloor n_0,n_1,n_2,\dots \rceil_{a,b},
\]
where $|n_k|$ is the
number of iterates under $f_{a,b}$ in between successive visits to $[a,b)$, and the sign of $n_k$ shows whether the iterates are to the left or right of $[a,b)$. This is explained in more detail, with the same notations, in~\cite[Section~2]{KU-JMD}. We refer to $f_{a,b}$ as a ``slow'' continued fraction map, in contrast to a Gauss-like map (the first return of $f_{a,b}$
to the interval $[a,b)$).
Several particular parameter choices correspond to well-studied continued fraction algorithms (see~\cite{KU-CWI,KU-JMD} for references): 
the case $(a,b) = (-1,1)$ corresponds to regular (plus) continued fractions with alternating signs for digits (it is also related to a method of symbolically coding the geodesic flow on the modular surface following Artin's pioneering work), and the case $(a,b) = (-\half,\half)$ gives the ``nearest-integer'' continued fractions considered first by Hurwitz.
These two cases play a pivotal role in~\Cref{sec:A and H}. The case $(a,b) = (-1,0)$ is also noteworthy, corresponding to the classical minus (also called backwards) continued fractions; this case will be mentioned again in \Cref{sec:further}.

% Larger list:
% \begin{itemize}
% 	\item $(-1,1)$ called Artin
% 	\item $(-\half,\half)$ nearest-integer continued fractions, first studied Hurwitz
% 	\item $(-1,0)$ classical minus (also called backwards or Gauss) continued fractions
% 	\item $(\frac{1-\sqrt5}2, \frac{-1+\sqrt5}2)$ dual-Hurtwiz (also called grotesque)
% 	\item family $(b-1,b)$ is conceptually similar to Japanese or $\alpha$ continued fractions 
% 	\item family $(-b,b)$ is related to factors of Japanese continued fraction maps
% \end{itemize}
%{\color{red} An extensive literature exists about absolutely continuous invariant measures and measure-theoretic entropy for Gauss-like maps associated to various continued fraction algorithms.}
The notion of topological entropy was introduced by Adler, Konheim, and McAndrew in~\cite{AKM}. Their definition used covers and applied to compact Hausdorff spaces; Dinaburg~\cite{D70} and Bowen~\cite{B7173} gave definitions involving distance functions and separated sets, which are often more suitable for calculation. While these formulations of topological entropy were originally intended for continuous maps acting on compact spaces, Bowen's definition can actually be applied to piecewise continuous, piecewise monotone maps on an interval, as explained in~\cite{MZ}. The most convenient definition of topological entropy for piecewise continuous piecewise monotone maps is
\[ h_{\rm top}(f) = \lim_{n\rightarrow \infty} \frac{\log(\# \text{ of laps of } f^n)}{n}, \]
where a \emph{lap} is a maximal interval of monotonicity for a function~\cite{ALM, MSz}.
In~\cite{MZ} it is shown that this agrees with Bowen's definition of topological entropy.
When a map is \emph{Markov}, i.e., it admits a finite Markov partition (see~\cite[Chapter~1.9]{KH}), its topological entropy can be found explicitly as the log of the spectral radius (the maximum absolute value of the eigenvalues) of the associated transition matrix.

As a one-dimensional map on $\bar\R$, each $f_{a,b}$ is piecewise continuous and piecewise monotone. The map 
\[ \label{compactify}
	\comp(x) := \frac{x}{1+|x|}%, \qquad \comp^{-1}(x) = \frac{x}{1-|x|}
\]
is a homeomorphism from $\bar\R$ to $[-1,1]/{\sim}$ with $\pm1$ identified; for convenience, we will write only $[-1,1]$ and deal with interval maps, although the results and proofs could all be done on a circle.
To make our notation more uniform, we conjugate the standard generators 
%$T,S,T^{-1}$ from \eqref{eqn:generators} 
\[ \label{eqn:generators}
	T(x) := x+1
	\qquad\text{and}\qquad
	S(x) := -\frac1x
\]
of the modular group $PSL(2,\mathbb Z)$ to
\[ 
	\tilde T := \comp \circ T \circ \comp^{-1}
	\qquad\text{and}\qquad
	\tilde S := \comp \circ S \circ \comp^{-1}
\]
(see \Cref{fig:three functions} on page~\pageref{fig:three functions}) and conjugate our continued fraction map $f_{a,b}:\bar\R\to\bar\R$ to the map $\tilde f_{a,b}:[-1,1] \to [-1,1]$,
\< \label{eqn:fab tilde}
	{\tilde f}_{a,b}(x) 
	:= \comp \circ f_{a,b} \circ \comp^{-1}(x) 
	= \begin{cases}
		\tilde T(x) &\text{if } -1 \le x < \tfrac{a}{1-a} \\
		\tilde S(x) &\text{if } \tfrac{a}{1-a} \le x < \tfrac{b}{1+b} \\
		\tilde T^{-1}(x) &\text{if } \tfrac{b}{1+b} \le x \le 1,
	\end{cases}
\>
thus obtaining a piecewise monotone map with two discontinuity points $\comp(a)$ and $\comp(b)$, see \Cref{fig:f and ftilde} (note that on the right, $\comp(a)=-\tfrac49$ and $\comp(b)=\tfrac27$).

\begin{figure}[htb]
\begin{tikzpicture}[scale=1.9]
	\tikzstyle{mygraph}=[blue,thick];
	\def\Inf{4}
	
	\begin{scope}[scale=1/3]
		\draw [thin,black] (-\Inf-0.2,0) -- (\Inf+0.2,0) (0,-\Inf-0.2) -- (0,\Inf+0.2);
		\foreach \x in {-\Inf, ..., -2, 1, 2, ..., \Inf} \draw (\x,0.1) -- (\x,-0.1) node [below] {\tiny$\x$};
		\draw (-0.8,0.1) -- (-0.8,-0.1) node [below] {\tiny$-\tfrac45$};
		\draw (0.4,0.1) -- (0.4,-0.1) node [below] {\tiny$\tfrac25$};
		
	  	\foreach \y in {-\Inf, ..., -3, -2, -1, 1, 2, 3, ..., \Inf} \draw (0.1,\y) -- (-0.1,\y) node [left] {\tiny$\y$};
		%\draw (0.1,-5/2) -- (-0.1,-5/2) node [left] {\tiny$-5/2$};
		%\draw (0.1,5/4) -- (-0.1,5/4) node [left] {\tiny$5/4$};
		
		\draw [mygraph,<-] (-\Inf,-\Inf+1) -- (-0.8,-0.8+1);
		\draw [mygraph,domain=-0.8:-1/\Inf,samples=25,->] plot (\x, -1/\x);
		\draw [mygraph,domain=1/\Inf:0.4,samples=25,<-] plot (\x, -1/\x);
		\draw [mygraph,->] (0.4,0.4-1) -- (\Inf,\Inf-1);
		
		\draw (0,\Inf+0.2) node [above] {$f_{-4/5,2/5}$ (original)};
	\end{scope}
	
	\begin{scope}[xshift=3.5cm,scale=1.25]
		\draw [thin,black,opacity=0.5] (-1,0) -- (1,0) (0,-1) -- (0,1);
		\draw (-1,-1) rectangle (1,1);
		\foreach \x/\t in {-1/-1, -0.444/{-\tfrac49}, 0.285714286/{\tfrac27}, 1/1} \draw (\x,-0.95) -- (\x,-1) node [below] {\tiny$\t$};
		\foreach \y/\t in {-1/-1, -0.714285714/{-\tfrac57}, 0.555/{\tfrac59}, 1/1} \draw (-0.95,\y) -- (-1,\y) node [left] {\tiny$\t$};
		
		\draw [mygraph,domain=-1:-1/2,samples=25] plot (\x, -2-1/\x);
		\draw [mygraph,domain=-1/2:-4/9,samples=25] plot (\x, {(2*\x+1)/(2+3*\x)});
		\draw [mygraph] (-4/9,5/9) -- (0,1) (0,-1) -- (2/7,-5/7);
		\draw [mygraph,domain=2/7:1/2,samples=25] plot (\x, {(2*\x-1)/(2-3*\x)});
		\draw [mygraph,domain=1/2:1,samples=25] plot (\x, 2-1/\x);
	\end{scope}
		\draw (3.5cm,{(\Inf+0.2)/3}) node [above] {${\tilde f}_{-4/5,2/5}$ (compact)};
\end{tikzpicture}
\caption{Plots of $f_{a,b}$ and ${\tilde f}_{a,b}$ for $a=-\tfrac45, b=\tfrac25$}
\label{fig:f and ftilde}
\end{figure}
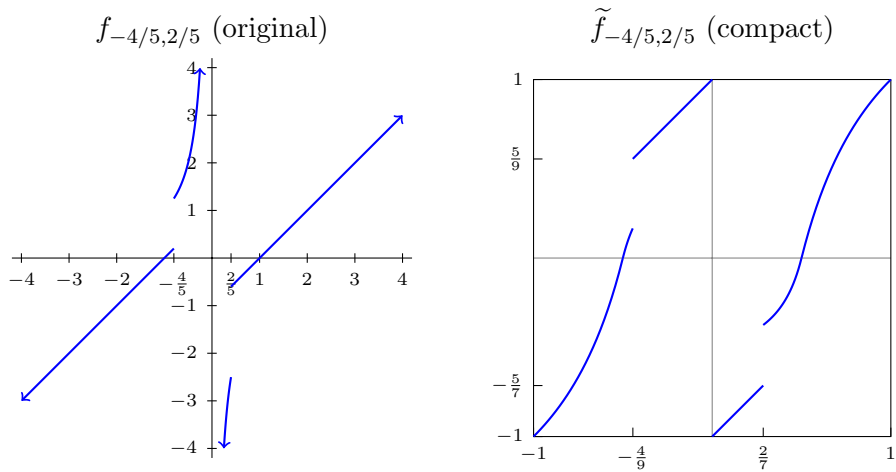

In this paper, we prove the following ``entropy locking" result:
\begin{thm} \label{main}
For any $(a,b)\in \Square =[-1,-\half]\times [\half,1] \subset \Params$, the topological entropy of ${\tilde f}_{a,b}$ (and, therefore, of $f_{a,b}$) is $\log(\frac{1+\sqrt5}2)$.
\end{thm}

The ``golden square" $\Square = [-1,-\half]\times [\half,1] \subset \Params$ is highlighted in Figure~\ref{fig:triangle}. Note that this subset contains (uncountably many) parameters for which $\tilde f_{a,b}$ does \textit{not} admit a Markov partition, and our entropy formula holds for these maps as well. Also note that two maps from this family are not necessarily topologically conjugate to each other.
\begin{figure}[hbt]
	\begin{tikzpicture}[scale=2.75]
		\def\Inf{1.5} \def\sz{\scriptsize}
		\drawP[thick]{black!25}{1.5}
		\begin{scope}
		\draw [orange, thick, fill=yellow!75!orange] (-1, 1/2) rectangle (-1/2, 1);
		\clip (-1, 1/2) rectangle (-1/2, 1);
		\foreach \x in {-1.2,-1.17,...,-0.5} \draw [orange] (\x,0.5) -- +(0.15,0.5);
		\end{scope}
		\draw (0,\Inf) node [above] {\sz$b$} -- (0,0) -- (-\Inf,0) node [left] {\sz$a$}; 
		
		\draw (-1,0) node [below] {\sz$-1$};
		\draw (0,1) node [right] {\sz$1$};
		\draw (0,0) node [right] {\sz$0$} node [below] {\sz$0$};
		
		%\draw [white,thick] (0,0.33) -- (0,0.67); % to remove part of b-axis
		\fill (-1,1) circle (0.8pt) node [left] {$(-1,1)$};
		\fill (-1/2,1/2) circle (0.8pt) node [right] {$(-\half,\half)$};
		%\draw [fill] (-1,0) circle (1pt) node [left] {$(-1,0)$};
	\end{tikzpicture}
	\caption{The parameter space $\Params$, with the ``golden square''~$\Square$ shaded}
	\label{fig:triangle}
\end{figure}
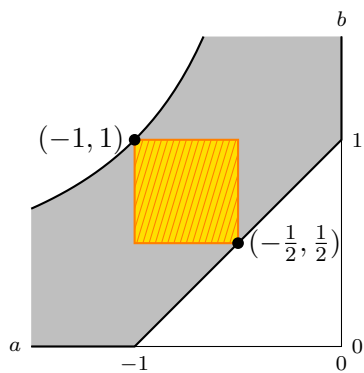

\begin{remark}
For a special family of piecewise affine maps with one discontinuity point, an entropy locking phenomenon was investigated by Bruin, Carminati, Marmi, and Profeti~\cite{Bruinetal} and by Cosper and Misiurewicz~\cite{CM}, both following numerical simulations from Botella-Soler et al.~\cite{Botella-Soler-et-al}. Our maps $\smash{\tilde f_{a,b}}$ are piecewise monotone and have two discontinuity points, so their methods do not readily apply here.
\end{remark}

\begin{remark}
In~\cite[Sections 6-7]{KU-ETDS}, the authors obtained an absolutely continuous invariant probability measure for the first return (Gauss-like) map of $f_{a,b}$ to $[a,b)$. With respect to this measure, the entropy of the Gauss-like map is $\frac1{K_{a,b}}\frac{\pi^2}{3}$~\cite[Theorem 6.2]{KU-ETDS}, where $K_{a,b}$ is the measure of the domain of the natural extension (which is finite). By lifting this measure to $\bar\R$, one obtains an infinite invariant measure for $f_{a,b}:\bar\R \to \bar\R$, so the classical notion of measure-theoretic entropy does not apply to $f_{a,b}$. It is an almost immediate consequence that the~``Krengel entropy'' (see~\cite{Krengel}) of $f_{a,b}$ is $\frac{\pi^2}{3}$ for all $(a,b) \in \Params$.
\end{remark}

\section{Proof of main result} \label{sec:proof of main}

The proof of \Cref{main} uses results about conjugacy to piecewise continuous maps with constant slope (see~\cite{AM}). Specifically, we prove that the maps $\tilde f_{-1,1}$ and $\tilde f_{-1/2,1/2}$ are conjugate to piecewise linear maps with constant slope $\frac{1+\sqrt5}2$ using \textit{the same conjugacy}. We then show that this conjugacy will also conjugate any $\tilde f_{a,b}$ with $(a,b) \in [-1,-\half] \times [\half,1]$ to a map with constant slope $\frac{1+\sqrt5}2$. A similar argument was first used by the authors in~\cite{AKU-Rigidity} for maps related to co-compact Fuchsian groups. An important ingredient of this approach is a symbolic ``recoding'' process, addressed in \Cref{recoding} below. The recoding in this paper turns out to be less intricate than the corresponding recoding in~\cite[Appendix~A]{AKU-Rigidity}.

\subsection{Conjugation to maps of constant slope}

In~\cite{Parry66}, following his seminal work~\cite{Parry64}, Parry proved that a continuous, piecewise monotone, topologically transitive interval\linebreak[4] Markov map with positive topological entropy is conjugate to a constant slope map. In~\cite{AM}, following~\cite{ALM}, Alsed\`a and Misiurewicz generalized this to piecewise continuous, piecewise monotone interval maps that are not necessarily Markov. For the present paper, we need only the original results of Parry.

%\begin{thm}[\cite{AM}] \label{thmAM} Let $I$ be a compact interval and let $g:I\to I$ be a piecewise monotone, piecewise continuous, topologically transitive map with positive topological entropy $h>0$. There exists a unique increasing homeomorphism $\psi:I \to I$ conjugating $g$ to a piecewise continuous map with constant slope $e^{h}$. \end{thm}

\begin{thm}[\cite{Parry66}] \label{thmP} Let $I$ be a compact interval and let $g:I\to I$ be a piecewise continuous, piecewise monotone, (strongly) transitive Markov map with positive topological entropy \mbox{$h>0$}. There exists an increasing homeomorphism $\psi:I \to I$ conjugating $g$ to a piecewise continuous map with constant slope $e^{h}$. \end{thm}

In the case where $\tilde f_{a,b}$ is Markov, the conjugacy $\psi_{a,b}$ can be obtained by the classical construction due to Parry~\cite{Parry66} and used in the proof of~\cite[Lem\-ma~5.1]{AM}. We define the probability measure $\rho_{a,b}$ on the shift space $X_{a,b} \subset \{1,...,N\}^\N$ as follows: let $\lambda,v$ (possibly depending on $a,b$) be the maximal eigenpair for the Markov transition matrix $M_{a,b}$; for an $(a,b)$-admissible finite sequence $(\omega_0,...,\omega_n)$ (that is, for which $(M_{a,b})_{\omega_i,\omega_{i+1}} = 1$ for $i=0,...,n-1$), we denote a symbolic cylinder of rank $n+1$ as
\[ C_{a,b}(\omega_0,\dots,\omega_n) := \big\{\, \omega' \in X_{a,b} \,|\, \omega'_i = \omega_i ~\forall~ 0 \le i \le n \,\big\} \]
and define the measure $\rho_{a,b}$ of this cylinder to be
\< \label{rho defn}
	\rho_{a,b}\big(C_{a,b}(\omega_0,\dots,\omega_n)\big) = \frac{v_{\omega_n}}{\lambda^n}. 
\>
The measure $\rho_{a,b}$ is equivalent to the shift-invariant ``Parry measure'' (the measure of maximal entropy; see~\cite{Parry64,Parry66}).
The measure $\rho_{a,b}$ is not shift-invariant but has the ``expanding property''
\[
\rho_{a,b}(\sigma_{a,b}(C))=\lambda \cdot \rho_{a,b}(C)
\]
for all cylinders $C$ on $(X_{a,b}, \sigma_{a,b})$.

Using the measure $\rho_{a,b}$, one constructs the push-forward Borel probability measure $\rho'_{a,b}$ on~$[-1,1]$ given by
\[ \label{mu'} \rho'_{a,b}(E) = \rho_{a,b}\big( \phi_{a,b}^{-1}(E) \big) \qquad\text{for Borel $E\subset [-1,1]$}, \]
where $\phi_{a,b}: X_{a,b} \to [-1,1]$ is the (essentially bijective) symbolic coding map, that is, $\phi_{a,b}(\omega) = \bigcap_{i=0}^\infty {\tilde f}_{a,b}^{-i}(I_{\omega_i})$.
The conjugacy $\psi_{a,b}:[-1,1] \to [-1,1]$ is given by
\< \label{psi formula from Parry} \psi_{a,b}(x) := 
-1 + 2\cdot \rho'_{a,b}\big([-1,x]\big).
\>
The presence of $-1$ and $2$ in the formula \eqref{psi formula from Parry} comes from our use of $[-1,1]$ as the domain for $\tilde f_{a,b}$.

\subsection{Artin and Hurwitz parameters} \label{sec:A and H}

We consider two particular parameter choices: the ``Artin'' case $(a,b) = (-1,1)$ and the ``Hurwitz'' case $(a,b) = (-\tfrac12,\tfrac12)$. From now on, we abbreviate $\tilde f_A = \tilde f_{-1,1}$, $\psi_A = \psi_{-1,1}$, etc., and $\tilde f_H = \tilde f_{-1/2,1/2}$, $\psi_H = \psi_{-1/2,1/2}$, etc. Additionally, the color red is used for components of Figures~\ref{fig:fA and fH}, \ref{fig:Markov graphs}, \ref{fig:intervals}, and \ref{fig:overlap} corresponding to~Artin, while green is used for Hurwitz.

\newcommand\drawWithGrid[2][]{
		\begin{scope}[#1]
		\foreach \x in {-1, -0.667, -0.5, -0.333, 0, 0.333, 0.5, 0.667, 1} \draw [thin,black!33] (\x,-1.1) -- (\x,1) (-1.1,\x) -- (1,\x);
	  	
		#2
		
		\draw (-1,-1) rectangle (1,1);
		\foreach \k/\x in {1/-0.833, 2/-0.583, 3/-0.417, 4/-0.167, 5/0.167, 6/0.417, 7/0.583, 8/0.833} \draw (\x,-1) node [below] {\scriptsize$I_\k$} (-1,\x) node [left] {\scriptsize$I_\k$};
	\end{scope}
}
\begin{figure}[tbh]
\vspace*{-5pt}
\begin{tikzpicture}[scale=2.4]
	\tikzstyle{mygraph}=[artincolor,very thick];
	\drawWithGrid{
		\draw [mygraph,domain=-1:-1/2,samples=25] plot (\x, -2-1/\x);
		
		\draw [mygraph] (-1/2,1/2) -- (0,1) (0,-1) -- (1/2,-1/2);
		
		\draw [mygraph,domain=1/2:1,samples=25] plot (\x, 2-1/\x);
		
		\draw (0,1) node [above] {${\tilde f}_A$};
	}
	
	\tikzstyle{mygraph}=[hurwitzcolor,very thick];
	\drawWithGrid[xshift=3cm]{
		\draw [mygraph,domain=-1:-1/2,samples=25] plot (\x, -2-1/\x);
		
		\draw [mygraph,domain=-1/2:-1/3,samples=25] plot (\x, {(2*\x+1)/(2+3*\x)});
		\draw [mygraph] (-1/3,2/3) -- (0,1) (0,-1) -- (1/3,-2/3);
		\draw [mygraph,domain=1/3:1/2,samples=25] plot (\x, {(2*\x-1)/(2-3*\x)});
		
		\draw [mygraph,domain=1/2:1,samples=25] plot (\x, 2-1/\x);
		
		\draw (0,1) node [above] {${\tilde f}_H$};
	}
\end{tikzpicture}
\caption{Plots of ${\tilde f}_A = {\tilde f}_{-1,1}$ and ${\tilde f}_H = {\tilde f}_{-1/2,1/2}$ with their (shared) Markov partition of~$[-1,1]$}
\label{fig:fA and fH}
\end{figure}
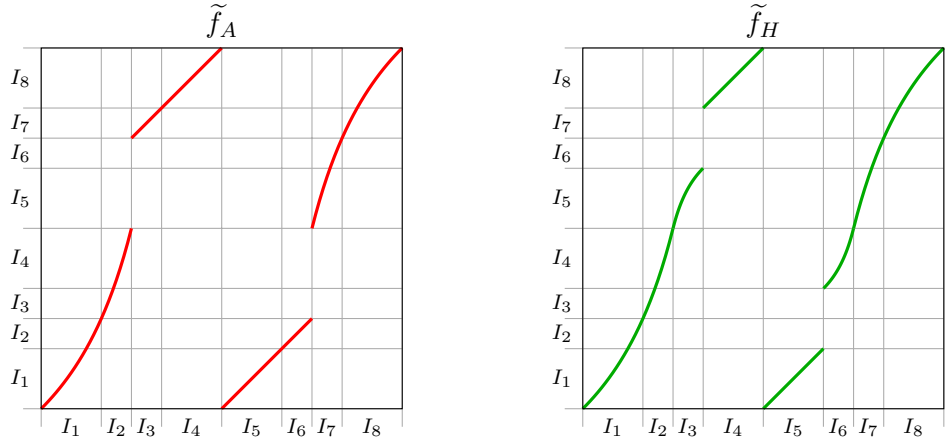

The maps ${\tilde f}_A$ and ${\tilde f}_H$ are each piecewise monotone, piecewise continuous, topologically transitive, and Markov with respect to \textit{the same partition} $\{I_1,...,I_8\}$ of the interval $[-1,1]$ (see \Cref{fig:fA and fH}):
\begin{align*}
I_1 &= [-1,-\tfrac23], &
I_2 &= [-\tfrac23,-\tfrac12], &
I_3 &= [-\tfrac12,-\tfrac13], &
I_4 &= [-\tfrac13,0], \\
I_5 &= [0,\tfrac13], &
I_6 &= [\tfrac13,\tfrac12], &
I_7 &= [\tfrac12,\tfrac23], &
I_8 &= [\tfrac23,1].
\end{align*}
The associated Markov diagrams are shown in \Cref{fig:Markov graphs}.
\begin{figure}[hbt]
	\newcommand\markovgraph[4][black]{
		\begin{tikzpicture}[scale=1.75]
		\foreach \k/\a/\r in {1/150/1, 2/90/0.5, 3/180/0.45, 4/30/1, 5/-150/1, 6/0/0.45, 7/-90/0.5, 8/-30/1} \node [draw, circle] at (\a:\r) (\k) {$\k$};
		%{1/150/1, 2/90/1, 3/180/0.4, 4/30/1, 5/-150/1, 6/0/0.4, 7/-90/1, 8/-30/1} for hex
	
		\foreach \j/\k in {1/2,2/4,4/8,8/7,7/5,5/1, 2/3, 7/6} \draw [->, thick] (\j) -- (\k);
		\draw [->, thick] (1) to [loop left] (1);
		\draw [->, thick] (8) to [loop right] (8);
		
		\draw [->, thick, #1] (3) -- (#3);
		\draw [->, thick, #1] (6) -- (#4);
		\node [above of=2, node distance=1.75em] {#2};
		\end{tikzpicture}
	}
	\qquad
	\markovgraph[artincolor]{Artin}{7}{2}
	\qquad
	\markovgraph[hurwitzcolor]{Hurwitz}{5}{4}
	\caption{Markov structure of Artin (left) and Hurwitz (right) admissibility}
	\label{fig:Markov graphs}
\end{figure}
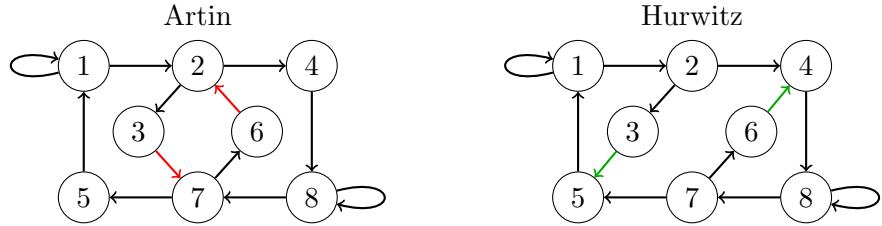

From the Markov diagrams, we construct the pair of $8\times 8$ transition matrices
\begin{equation}
	M_A=\begin{pmatrix}
	1 & 1 & 0 & 0 & 0 & 0 & 0 & 0\\
	0 & 0 & 1 & 1 & 0 & 0 & 0 & 0\\
	0 & 0 & 0 & 0 & 0 & 0 & \mathbf{1} & 0\\
	0 & 0 & 0 & 0 & 0 & 0 & 0 & 1\\
	1 & 0 & 0 & 0 & 0 & 0 & 0 & 0\\
	0 & \mathbf{1} & 0 & 0 & 0 & 0 & 0 & 0\\
	0 & 0 & 0 & 0 & 1 & 1 & 0 & 0\\
	0 & 0 & 0 & 0 & 0 & 0 & 1 & 1\\
	\end{pmatrix},
	\quad M_H=\begin{pmatrix}
	1 & 1 & 0 & 0 & 0 & 0 & 0 & 0\\
	0 & 0 & 1 & 1 & 0 & 0 & 0 & 0\\
	0 & 0 & 0 & 0 & \mathbf{1} & 0 & 0 & 0\\
	0 & 0 & 0 & 0 & 0 & 0 & 0 & 1\\
	1 & 0 & 0 & 0 & 0 & 0 & 0 & 0\\
	0 & 0 & 0 & \mathbf{1} & 0 & 0 & 0 & 0\\
	0 & 0 & 0 & 0 & 1 & 1 & 0 & 0\\
	0 & 0 & 0 & 0 & 0 & 0 & 1 & 1\\
	\end{pmatrix}.
\end{equation}
Notice that the two matrices are very similar, the only difference being that the transitions $I_3\rightarrow I_7$ and $I_6\rightarrow I_2$ in $M_A$ are replaced by $I_3\rightarrow I_5$ and $I_6\rightarrow I_4$ in $M_H$.

By direct computation, the characteristic polynomials of $M_A$ and $M_H$ are 
\[ (x^2 - x - 1) (x^2 - x + 1) x^4 \quad\text{and}\quad (x^2-x-1) (x^2-x+1) (x^4-1), \] 
respectively, and so both have the same dominant eigenvalue \[ \lambda = \frac{1+\sqrt5}2. \] The corresponding (right probability) eigenvector for both matrices is
\begin{equation}\label{eigenvector}
	v = \frac1{6\lambda+4} (\lambda\!+\!1,\; \lambda,\; 1,\; \lambda,\; \lambda,\; 1,\; \lambda,\; \lambda\!+\!1)
.\end{equation}

To aid in later proofs, we denote  $A_k={\tilde f}_A|_{I_k}$ and $H_k={\tilde f}_H|_{I_k}$. The following lemma can be proven by looking at the graphs in \Cref{fig:fA and fH} or by a careful analysis of \eqref{eqn:fab tilde} on the intervals $I_1,...,I_8$:
\begin{lem} \label{A and H maps}
	The maps $A_1, A_2, H_1, H_2, H_3$ coincide with $\tilde T$. The maps $A_3, A_4, A_5, A_6, H_4, H_5$ coincide with $\tilde S$. The maps $A_7, A_8, H_6, H_7, H_8$ coincide with~$\tilde T^{-1}$.
	In particular, if $k \notin \{3,6\}$, then $A_k = H_k$.
\end{lem}

By \Cref{thmP} there exists a (unique) increasing homeomorphism $\psi_A : [-1,1] \to [-1,1]$ conjugating ${\tilde f}_A$ to a map \[ \ell_A := \psi_A \circ {\tilde f}_A \circ \psi_A^{-1} \] with constant slope $\lambda=\frac{1+\sqrt5}2$, and there exists a (unique) increasing homeomorphism $\psi_H : [-1,1] \to [-1,1]$ conjugating ${\tilde f}_H$ to a map \[ \ell_H := \psi_H \circ {\tilde f}_H \circ \psi_H^{-1} \] also with constant slope $\frac{1+\sqrt5}2$.
We will prove that the maps $\psi_A$ and $\psi_H$ (each obtained by Parry's construction) coincide:

\begin{thm} \label{psi equality} For all $x \in [-1,1]$, $\psi_A(x) = \psi_H(x)$. \end{thm}

Equivalently, $\rho'_A(J) = \rho'_H(J)$ for all intervals $J \subset [-1,1]$. It is sufficient to take $J$ to be a cylinder interval: 
%For our proof, we need to find a relationship between the cylinder intervals of the two maps $\tilde f_A$ and $\tilde f_H$. 
given an $(a,b)$-admissible sequence $\omega=(\omega_0,\omega_1,\dots,\omega_n)$ with $\omega_i \in \{1,\dots,8\}$, we define the corresponding \emph{$(a,b)$-cylinder interval} of rank $n+1$ as
\< \label{eqn:cylinder}
\begin{split}
	I_{a,b}{(\omega_0,\omega_1,\dots,\omega_n)}
	:=\: & I_{\omega_0} \cap {\tilde f}_{a,b}^{-1}(I_{\omega_1}) \cap \cdots \cap {\tilde f}_{a,b}^{-n}(I_{\omega_n}) \\
	=\: & I_{\omega_0} \cap {\tilde f}_{a,b}^{-1}(I_{a,b}(\omega_1,\dots,\omega_n)).
\end{split}
\>

If every $A$-admissible word $\omega = (\omega_0,...,\omega_n)$ had a corresponding $H$-admissible $\tau=(\tau_0,...,\tau_n)$ with $I_A(\omega) = I_H(\tau)$ and $v_{\omega_n} = v_{\tau_n}$ then \Cref{psi equality} would be proven. In fact, we do not need to show $I_A(\omega) = I_H(\tau)$ for \textit{all} $A$-admissible $\omega$, as we now explain.

If $\omega_n\in\{3,4,5,6\}$, then we use the unique Markov transitions $3 \to 7$, $4\to 8$, $5\to 1$ and $6\to 2$ to instead consider
\begin{align*}
    I_A(\omega_0,...,\omega_{n-1},3) = I_A(\omega_0,...,\omega_{n-1},3,7), \\
    I_A(\omega_0,...,\omega_{n-1},4) = I_A(\omega_0,...,\omega_{n-1},4,8), \\
    I_A(\omega_0,...,\omega_{n-1},5) = I_A(\omega_0,...,\omega_{n-1},5,1), \\
    I_A(\omega_0,...,\omega_{n-1},6) = I_A(\omega_0,...,\omega_{n-1},6,2).
\end{align*}
Therefore we can assume $\omega_n\notin\{3,4,5,6\}$.

\begin{lem} \label{recoding}
	If $\omega = (\omega_0,...,\omega_n)$ is $A$-admisisble and $\omega_n\in\{1,2,7,8\}$, then there exists an $H$-admissible word $\tau = (\tau_0,...,\tau_n)$ such that $\omega_0 = \tau_0$ and $I_A(\omega) = I_H(\tau)$. Moreover, if $\omega_n \in \{1, 8\}$, then $\omega_n = \tau_n$; if $\omega_n=2$, then $\tau_n\in\{2,4\}$; and if $\omega_n=7$, then $\tau_n\in\{5,7\}$.
\end{lem}

\begin{proof}%[Proof of \Cref{recoding}]

Recall the notation $A_k$ and $H_k$ from \Cref{A and H maps}. From the monotonicity of $A_k$ and $H_k$, we can avoid the intersection in the definition \eqref{eqn:cylinder} and instead calculate
\begin{align*}
	I_A(\omega_0,\omega_1,\dots,\omega_n)
	%= I_{\omega_0} \cap {\tilde f}_A^{-1}(I_A(\omega_1,\dots,\omega_n))
	= A_{\omega_0}^{-1}(I_A(\omega_1,\dots,\omega_n)) 
	\quad \text{if  $(\omega_0,\omega_1,\dots,\omega_n)$ is $A$-admissible} 
\end{align*}
and 
\begin{align*}
	I_H(\tau_0,\tau_1,\dots,\tau_n)
	= H_{\tau_0}^{-1}(I_H(\tau_1,\dots,\tau_n))
	\quad \text{if $(\tau_0,\tau_1,\dots,\tau_n)$ is $H$-admissible.}
\end{align*}
With these observations, in \Cref{fig:intervals} we have the following correspondence between the rank two cylinder intervals of the two maps.
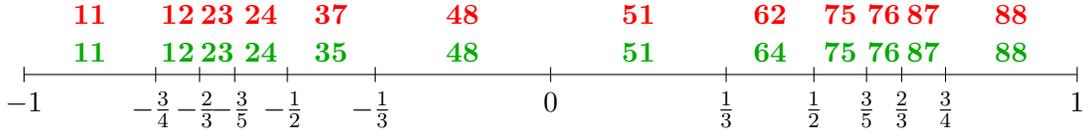
\begin{figure}[thb]
\begin{tikzpicture}[xscale=7]
	\draw (-1,0) -- (1,0);
	\foreach \x in {-1,0,1} \draw (\x,0.1) -- (\x,-0.1) node [below] {$\x$};
	\foreach \a/\b in {1/3,1/2,3/5,2/3,3/4} \draw (\a/\b,0.1) -- (\a/\b,-0.1) node [below] {$\tfrac{\a}{\b}$} (-\a/\b,0.1) -- (-\a/\b,-0.1) node [below] {$-\tfrac{\a}{\b}$\;};
	
	\foreach \a/\b/\x in {1/1/-0.875, 1/2/-0.708, 2/3/-0.633, 2/4/-0.550, 3/5/-0.417, 4/8/-0.167, 5/1/0.167, 6/4/0.417, 7/5/0.550, 7/6/0.633, 8/7/0.708, 8/8/0.875} \draw (\x,0.3) node {\bf\color{hurwitzcolor}\a\b}; % -1/2,1/2
	
	\foreach \a/\b/\x in {1/1/-0.875, 1/2/-0.708, 2/3/-0.633, 2/4/-0.550, 3/7/-0.417, 4/8/-0.167, 5/1/0.167, 6/2/0.417, 7/5/0.550, 7/6/0.633, 8/7/0.708, 8/8/0.875} \draw (\x,0.8) node {\bf\color{artincolor}\a\b}; % -1,1
\end{tikzpicture}
\caption{Rank two cylinder intervals for ${\tilde f}_A$ (red) and ${\tilde f}_H$ (green) coincide}
\label{fig:intervals}
\end{figure}

The matching $I_A(1,1) = I_H(1,1)$ and $I_A(1,2) = I_H(1,2)$ and similarly for 
%$23$, $24$, $48$, $51$, $75$, $76$, $87$, and $88$
$(2,3)$, $(2,4)$, $(4,8)$, $(5,1)$, $(7,5)$, $(7,6)$, $(8,7)$, and $(8,8)$ 
are all trivial due to \Cref{A and H maps}.
One can also check directly that
\[
I_A(3,7)=I_H(3,5)=I_3 \quad\text{and}\quad I_A(6,2)=I_H(6,4)=I_6.
\]

We have the following identities:
\begin{align*}
	I_A(\mathbf3,7,5,\mathbf1) &= I_H(\mathbf3,5,1,\mathbf1), \\
	I_A(\mathbf3,7,6,\mathbf2) &= I_H(\mathbf3,5,1,\mathbf2), \\
	I_A(\mathbf6,2,3,\mathbf7) &= I_H(\mathbf6,4,8,\mathbf7), \\
	I_A(\mathbf6,2,4,\mathbf8) &= I_H(\mathbf6,4,8,\mathbf8).
\end{align*}
These all follow from the classical relationship $(\tilde T\tilde S)^3=\mathrm{Id}$ on the generators of $\mathrm{SL}(2,\Z)$. An equivalent formulation is $\tilde S\tilde T \tilde S = \tilde T^{-1}\tilde S\tilde T^{-1}$. 
%, and with this we can prove the first relation above:
% \[
% \begin{split}
% I_A(3,7,5,1)=A_{3}^{-1}A_7^{-1}A_5^{-1}(I_1)&=\tilde S\tilde T \tilde S(I_1)\\&= \tilde T^{-1}\tilde S\tilde T^{-1}(I_1)=H_{3}^{-1}H_5^{-1}H_{1}^{-1}(I_1)=I_H(3,5,1,1).
% \end{split}
% \]
By definition,
\[
I_A(3,7,5,1)=A_{3}^{-1}A_7^{-1}A_5^{-1}(I_1)=\tilde S\tilde T \tilde S(I_1)
\]
and
\[
I_H(3,5,1,1)
= H_{3}^{-1}H_5^{-1}H_{1}^{-1}(I_1)
= \tilde T^{-1}\tilde S\tilde T^{-1}(I_1);
\]
because $\tilde S\tilde T \tilde S = \tilde T^{-1}\tilde S\tilde T^{-1}$, we have $I_A(3,7,5,1) = I_H(3,5,1,1)$. The other three relations are proved similarly. Alternatively, this can be proved by tracking the image of $I_1$ under the relevant maps shown in \Cref{fig:three functions}.
\begin{figure}[htb]
	\begin{tikzpicture}[scale=1.9]
	\drawWithGrid{
		\tikzstyle{mygraph} = [very thick, blue];
		\draw [mygraph,domain=-1:-1/2,samples=25] plot (\x, {-2-1/\x});
		\draw [mygraph,domain=-1/2:0,samples=25] plot (\x,  {(1+2*\x)/(2+3*\x)});
		\draw [mygraph,domain=0:1,samples=50] plot (\x,  {1/(2-\x)});
		\draw (0,1) node [above] {$\tilde T$};
	}
	\drawWithGrid[xshift=2.75cm]{
		\draw [very thick, blue] (-1,0) -- (0,1) (0,-1) -- (1,0);
		\draw (0,1) node [above] {$\tilde S$};
	}
	\drawWithGrid[xshift=5.5cm]{
		\tikzstyle{mygraph} = [very thick, blue];
		\draw [mygraph,domain=-1:0,samples=50] plot (\x, {-1/(2+\x)});
		\draw [mygraph,domain=0:1/2,samples=25] plot (\x,  {(1-2*\x)/(-2+3*\x)});
		\draw [mygraph,domain=1/2:1,samples=25] plot (\x,  {2-1/\x});
		\draw (0,1) node [above] {$\tilde T^{-1}$};
	}
	\end{tikzpicture}
	\caption{Graphs of $\tilde T$ (left), $\tilde S$ (middle), and $\tilde T^{-1}$ (right), each with the Markov shared partition for Artin and Hurwitz maps}
	\label{fig:three functions}
\end{figure}
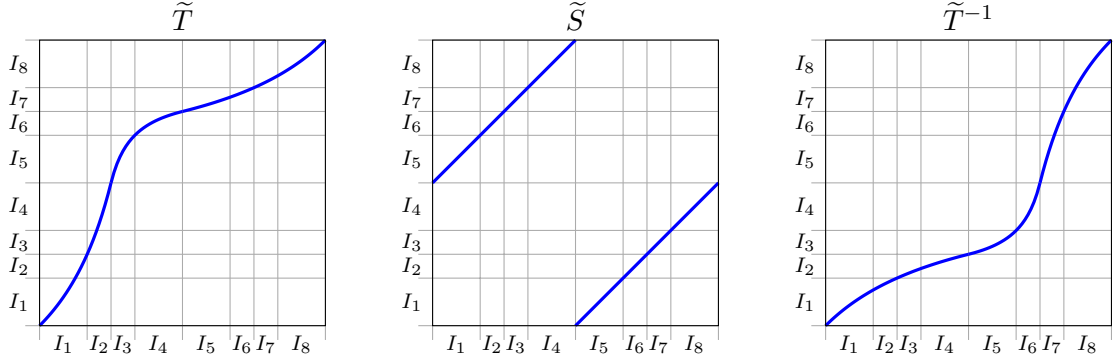

Notice that the recoding relations do not affect the first and fourth digit. This means that the recoding process will be localized to these words of length $4$, and it does not affect the symbols on either side of the block.

We refer to the four words
\[ 3751, \quad 3762, \quad 6237, \quad 6248 \]
as ``exceptional blocks''.
Our goal is to prove \Cref{recoding} by induction on the number $\ell$ of exceptional blocks that occur in $\omega$.

\textit{Base case.} If there are no exceptional blocks in $(\omega_0,\omega_1,\dots,\omega_n)$, then the sequence $(\omega_0,\omega_1,\dots,\omega_{n-2})$ will not contain the symbols $3$ or $6$ and so by the final statement of \Cref{A and H maps} we have 
\[
\begin{split}
I_A(\omega_0,\omega_1,\dots,\omega_n)&=A^{-1}_{\omega_0}\cdots A^{-1}_{\omega_{n-2}}(I_A(\omega_{n-1},\omega_n))\\
&=H^{-1}_{\omega_0}\cdots H_{\omega_{n-2}}^{-1}(I_H(\omega_{n-1},\tau_n))\\
&=I_H(\omega_0,\omega_1,\dots,\omega_{n-1},\tau_n),
\end{split}
\]
where $\tau_n$ is determined from the matching $I_A(\omega_{n-1},\omega_n)=I_H(\omega_{n-1},\tau_n)$ (see Figure~\ref{fig:intervals}).

\textit{Induction step.}
Now assume that the statement of \Cref{recoding} is true for any $A$-admissible finite sequence $\omega$ that contains $\ell>0$ exceptional blocks; we will prove that the statement holds for any word $(\omega_0,\omega_1,\dots,\omega_n)$ that contains $\ell+1$ exceptional blocks.

Let $k$ be the index where the first exceptional block appears. From the induction assumption,
\[
I_A(\omega_{k+3},\omega_{k+4},\dots,\omega_n)=I_H(\tau_{k+3},\tau_{k+4},\dots,\tau_n),
\]
where $\tau_{k+3}=\omega_{k+3}$; 
if $\omega_n \in \{1, 8\}$ then $\omega_n = \tau_n$. If $\omega_n=2$ then $\tau_n\in\{2,4\}$; and if $\omega_n=7$ then $\tau_n\in\{5,7\}$.
\pagebreak

Now, finally, we have
\begin{align*}
I_A(\omega)
&= I_A(\omega_0,\omega_1,\dots,\omega_{k-1},\textcolor{artincolor}{\omega_k,\omega_{k+1},\omega_{k+2},\omega_{k+3}},\dots,\omega_n) \\*
&=A^{-1}_{\omega_0}\cdots A^{-1}_{\omega_{k-1}}\textcolor{artincolor}{A^{-1}_{\omega_k}A^{-1}_{\omega_{k+1}}A^{-1}_{\omega_{k+2}}}(I_A(\omega_{k+3},\dots,\omega_n))\\
&= H^{-1}_{\omega_0}\cdots H_{\omega_{k-1}}^{-1}\textcolor{hurwitzcolor}{H^{-1}_{\omega_k}H^{-1}_{\tau_{k+1}}H^{-1}_{\tau_{k+2}}}(I_H(\omega_{k+3},\dots,\tau_n))\\
&= I_H(\omega_0,\omega_1,\dots,\omega_{k-1},\textcolor{hurwitzcolor}{\omega_k,\tau_{k+1},\tau_{k+2},\omega_{k+3}},\dots,\tau_n) 
\\* &=
I_H(\tau). \qedhere
\end{align*}
\end{proof}

With the recoding process (\Cref{recoding}) in place, we return to showing that $\psi_A =\psi_H$.

\begin{proof}[Proof of \Cref{psi equality}]
By the assumptions of \Cref{recoding}, it is enough to show that $\rho'_A(I_A(\omega)) = \rho'_H(I_A(\omega))$ for $\omega_n \in \{1,2,7,8\}$. Fix $\omega = (\omega_0,...,\omega_n)$ and let $\tau = (\tau_0,...,\tau_n)$ be the corresponding $H$-admissible word from \Cref{recoding}.

Since $\phi_A$ maps the cylinder interval $I_A(\omega)$ exactly to the symbolic cylinder $C_A(\omega)$, %%%we have
\[
	\rho'_A(I_A(\omega_0,\ldots,\omega_n)) = \frac{ v_{\omega_n}}{\lambda^n}.
\]
Because $I_A(\omega_0,\ldots,\omega_n) =I_H(\tau_0,\ldots,\tau_n)$, and $\phi_H$ maps the cylinder interval $I_H(\tau)$ exactly to the symbolic cylinder $C_H(\tau)$, in fact
\[
	\rho'_H(I_A(\omega_0,\ldots,\omega_n)) =\rho'_H(I_H(\tau_0,\ldots,\tau_n)) =\frac{v_{\tau_n}}{\lambda^n}. 
\]
% (OLD VERSION:) From the fact that $\tau_n=\omega_n$ if $\omega_n\in\{1, 8\}$; $\omega_n=2\Rightarrow \tau_n\in\{2,4\}$; $\omega_n=7\Rightarrow \tau_n\in\{5,7\}$ and \eqref{eigenvector}, we have $v_{\omega_n}=v_{\tau_n}$, so
% \[
%  \rho'_A(I_A(\omega_0,\ldots,\omega_n)) =  \rho'_H(I_A(\omega_0,\ldots,\omega_n)) 
% \]
The claim $\rho'_A(I_A(\omega)) = \rho'_H(I_A(\omega))$ is now equivalent to $v_{\omega_n} = v_{\tau_n}$, and this is easily checked using \eqref{eigenvector} and the final parts of \Cref{recoding}.
If $\omega_n \in \{1,8\}$ then $\omega_n = \tau_n$ (so trivially $v_{\omega_n} = v_{\tau_n}$). If $\omega_n = 2$ then $\tau_n \in \{2,4\}$, and since $v_2 = v_4$ this is also fine. Similarly, if $\omega_n = 7$ then $\tau_n \in \{5,7\}$ is sufficient because $v_5 = v_7$.

Having proven that $\rho'_A(I_A(\omega)) = \rho'_H(I_A(\omega))$ for a generating set of intervals $I_A(\omega)$, and using the definition \eqref{psi formula from Parry}, we conclude that $\psi_A = \psi_H$.
\end{proof}

\subsection{Proof of Theorem~\ref{main}}

Since $\psi_A = \psi_H$ by \Cref{psi equality}, we will now denote these two maps by simply $\psi$.
The map ${\tilde f}_H$ acts as $\tilde T$ on the interval $\comp([-\infty,-\half]) = [-1,-\third]$, so by construction, $\psi \circ \tilde T \circ \psi^{-1}$ is linear (with slope $\lambda = \frac{1+\sqrt5}2$, which we now refrain from repeating) on the interval $\psi([-1,-\third])$.
Similarly, ${\tilde f}_A$ acts as $\tilde S$ on $\comp([-1,1]) = [-\half,\half]$ and so $\psi \circ \tilde S \circ \psi^{-1}$ is linear on $\psi([-\half,\half])$.
And because ${\tilde f}_H$ acts as $\tilde T^{-1}$ on $[\third,1]$, we know that $\psi \circ \tilde T^{-1} \circ \psi^{-1}$ is linear on $\psi([\third,1])$. See~\Cref{fig:overlap}.
\begin{figure}[htb]
    \begin{tikzpicture}[scale=2.8]
        % Set the colors, etc., here:
        \tikzstyle{A}=[red!75!black, line width=1.33, dashed]
        \tikzstyle{H}=[green!90!black, line width=2.33]
        
        % Graphs
		\foreach \s in {1,-1} {
	    \draw [scale=\s, H] (-1,-1) -- (-0.236068, 0.236068);
	    \draw [scale=\s, H] (-0.236068, 1 - 0.381966) -- (0,1);
	    
	    \draw [scale=\s, A] (-1,-1) -- (-0.381966, 0);
	    \draw [scale=\s, A] (-0.381966, 0.381966) -- (0,1);
	    }
	    % Tick marks
	    \foreach \x/\t in {-1/-1,1/1} {
	    	\draw (\x,-1) node [below] {\tiny$\t$} -- +(0,0.05);
			\draw (-1,\x) node [left] {\tiny$\t$} -- +(0.05,0);
		}
		\foreach \x/\t in { -0.381966/{$\psi(\comp(-1))$},   -0.236068/{$\psi(\comp(-1/2))$},  0.236068/{$\psi(\comp(1/2))$}, 0.381966/{$\psi(\comp(1))$} } {
		    \draw (\x,-1-0.033) -- +(0,0.066);
		    \draw (\x+0.05,-1.02) node [left,rotate=45] {\tiny\t};
		}
		
		\draw (-1,-1) rectangle (1,1);
	\end{tikzpicture}
    \caption{Partially-overlapping graphs of $\psi \circ {\tilde f}_A \circ \psi^{-1}$ (red, dashed) and $\psi \circ {\tilde f}_H \circ \psi^{-1}$ (green)}
    \label{fig:overlap}
\end{figure}
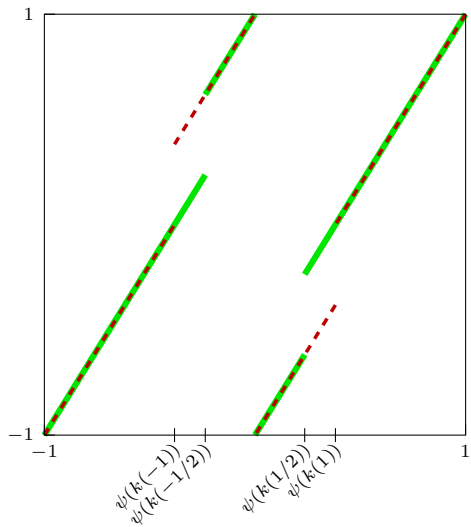

The map $\psi$ therefore satisfies the following four conditions:
	\begin{enumerate}[\quad(i)] 
	\addtolength\itemsep{2pt}
		\item \label{required1} for $x \in \psi([-1,-\third])$, $\psi(\tilde T(\psi^{-1}(x)) = \lambda x + c_1$;
		\item \label{required2} for $x \in \psi([-\half,0])$, $\psi(\tilde S(\psi^{-1}(x)) = \lambda x + c_2$;
		\item \label{required3} for $x \in \psi([0,\half])$, $\psi(\tilde S(\psi^{-1}(x)) = \lambda x + c_3$; \marginpar{\color{white}.} %%% For some reason, the \marginpar causes (iv) to be on the same page, while without \marginpar it is an orphan on the following page.
		\item \label{required4} for $x \in \psi([\third,1])$, $\psi(\tilde T^{-1}(\psi^{-1}(x)) = \lambda x + c_4$.
	\end{enumerate}

	\noindent In fact, one can calculate $c_1 = \lambda - 1, c_2 = 1, c_3 = -1, c_4 = 1-\lambda$, but these are not necessary for the proof.
	
	Let $(a,b) \in [-1,-\half]\times [\half,1]$ be arbitrary. The map ${\tilde f}_{a,b}$ acts as $\tilde T$ on the interval
	\[ [-1,\comp(a)] \subset [-1,-\third], \]
	and since $\psi \circ \tilde T \circ \psi^{-1}$ is linear on all of $\psi([-1,-\third])$ by (\ref{required1}), we have that $\psi \circ {\tilde f}_{a,b} \circ \psi^{-1}$ is linear on $\psi([-1,\comp(a)]) \subset \psi([-1,-\third])$.
	
	Similarly, $\psi \circ {\tilde f}_{a,b} \circ \psi^{-1}$ is linear on $\psi([\comp(a),0])$ because ${\tilde f}_{a,b}$ acts by $\tilde S$ on $[\comp(a),0] \subset [-\half,0]$ and, by (\ref{required2}), $\psi \circ \tilde S \circ \psi^{-1}$ is linear on all of all $\psi([-\half,0])$.
	Likewise, $\psi \circ {\tilde f}_{a,b} \circ \psi^{-1}$ is linear on $\psi([0,\comp(b)]) \subset \psi([0,\half])$ by (\ref{required3}) and on $\psi([\comp(b),1]) \subset \psi([-\third,1])$ by (\ref{required4}).
	
	Since ${\tilde f}_{a,b}$ is conjugate to a map with constant slope $\lambda$ on all of $[-1,1]$, we have $h_{\rm top}({\tilde f}_{a,b}) = \log(\lambda) = \log(\frac{1+\sqrt5}2)$ by \cite{MSz}.
\hfill$\square$

\section{Further remarks and open questions} \label{sec:further}

\subsection{Slow Gauss map} \label{rem:slow Gauss} Let $g:[0,\infty] \to [0,\infty]$ be the classical ``slow Gauss map'' 
\[ g(x) = \begin{cases} 1/x &\text{if } 0 \le x < 1 \\ x-1 &\text{if } x \ge 1, \end{cases} \]
which is closely related with a regular (plus) continued fraction expansion: for $0<x<1$,
\[
x=\cfrac1{n_1+\cfrac1{n_2+\cfrac1{n_3+\cdots}}}:=[0,n_1,n_2,n_3,\dots ]
\]
where the digits $n_k$ are the number of consecutive iterates under $g$ that are in $[1, \infty]$ between visits to $[0,1)$. Let $\tilde g:[0,1] \to [0,1]$ be the compactified version, $\tilde g = \comp \circ g \circ \comp^{-1}$, where $k:[0,\infty] \to [0,1]$ is $k(x) = \frac{x}{1+x}$. In~\cite{B79}, Bowen considered measure-theoretic properties of these two maps.

The map $\tilde g:[0,1] \to [0,1]$ is semi-conjugate to $\tilde f_{-1,1}:[-1,1]\to[-1,1]$ via the absolute value function ${\rm abs}:[-1,1]\to[0,1]$, that is $g \circ {\rm abs} = {\rm abs} \circ {\tilde f}_{-1,1}$, so $g$ is a topological factor of $\tilde f_{-1,1}$. Although in general the topological entropy of a factor is only less than or equal to the topological entropy of the map, in this case we have equality of topological entropies: the map $\tilde g$ has the Markov partition $\{I_1,I_2\} = \{[0,\half], [\half,1]\}$ with transition matrix $\Big(\begin{smallmatrix}\text{\normalsize0} & \text{\normalsize1} \\[2pt] \text{\normalsize1} & \text{\normalsize1} \end{smallmatrix}\Big)$, which immediately gives that \( h_{\rm top}(\tilde g) = \log(\frac{1+\sqrt5}2). \) The reason for the equality is
a simple relationship between the regular continued fraction expansion and a $(-1,1)$-continued fraction expansions: for $0<x\le 1$,
\[
x=[0,n_1,n_2,\dots ]=\lfloor 0,-n_1,n_2,-n_3,\dots \rceil_{-1,1}.
\]

\subsection{Conjectures about entropy}

Outside of the square $\Square = [-1,-\half] \times [\half,1]$, there are many unanswered questions about the behaviour of $h_{\rm top}(f_{a,b})$. Using Markov partitions, we can calculate explicit entropy values for many rational values $(a,b)$, and from these we have created \Cref{fig:3d}.

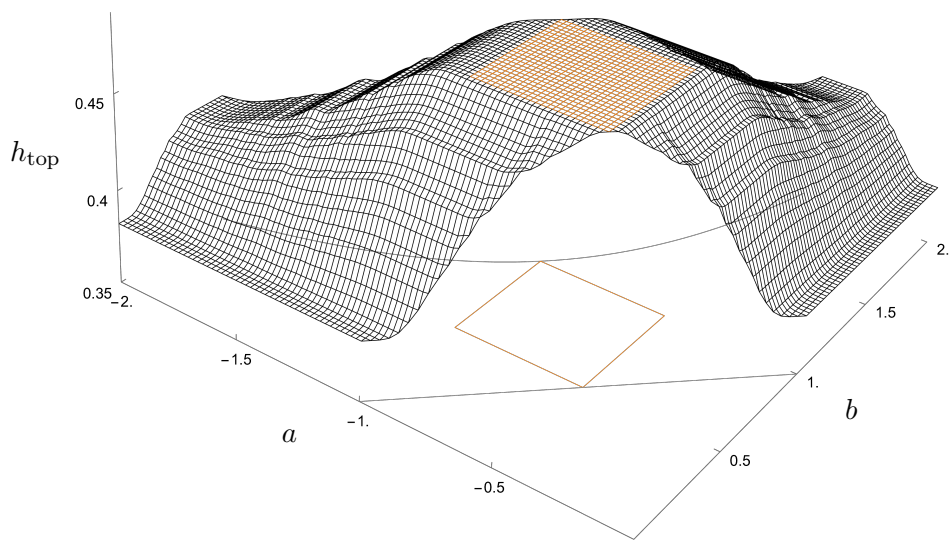
\begin{figure}[hbt]
	\begin{tikzpicture}[scale=0.9]
		\draw (0,-0.5) node {\includegraphics[width=12cm]{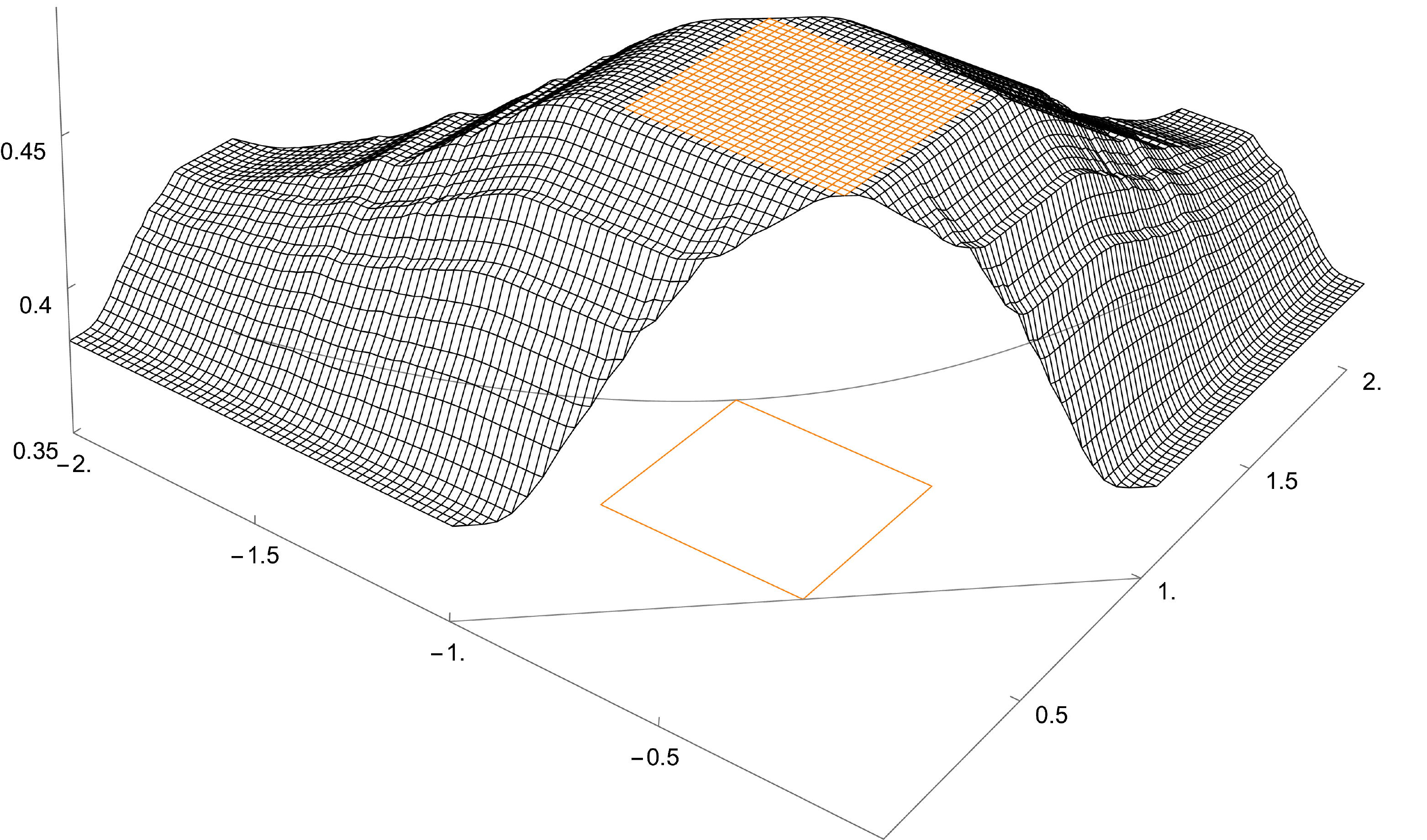}};
		\draw (-3.5,-2.9) node {$a$};
		\draw (4.8,-2.5) node {$b$};
		\draw (-7.25,1.25) node {$h_{\rm top}$};
	\end{tikzpicture}
	\caption{Plot of topological entropy, numerically (gold is proven)}
	\label{fig:3d}
\end{figure}

While in \Cref{sec:proof of main} we focused on the parameter choices $(-1,1)$ and $(-\half,\half)$, the map for the case $(a,b) = (-1,0)$ has also been studied independently, as it corresponds to classical backwards continued fractions~\cite{AF84,K96}. Additionally, according to our numerical tests, this parameter appears to give the minimum possible value for $h_{\rm top}(f_{a,b})$.

We can directly calculate the value
\[
	h_{\rm top}(f_{-1,0})
	= \log(\kappa)
	\approx 0.382,
\]
where $\kappa$ is the spectral radius of \[ M_{-1,0} = \begin{pmatrix} 1 & 1 & 0 & 0 \\ 0 & 0 & 0 & 1 \\ 0 & 1 & 0 & 0 \\ 0 & 0 & 1 & 1 \end{pmatrix} \] and satisfies $\kappa^3 - \kappa^2 - 1 = 0$. \medskip 

\begin{conjecture}[Flexibility] ~
	\begin{enumerate}[\quad(i)]
		\item If $(a,b) \in \Params$ then $\log(\kappa) \le h_{\rm top}(f_{a,b}) \le \log(\frac{1+\sqrt5}2)$.
		\item For any $h \in [\log(\kappa), \log(\frac{1+\sqrt5}2)]$, there exists $(a,b) \in \Params$ with $h_{\rm top}(f_{a,b}) = h$.
	\end{enumerate}
\end{conjecture}

\begin{conjecture}[Continuity and monotonicity] \label{cm} ~
    \begin{enumerate}[\quad(i)]
        \item The function $(a,b) \mapsto h_{\rm top}(f_{a,b})$ is continuous.
        \item For fixed $b \le \half$, the function $a \mapsto h_{\rm top}(f_{a,b})$ is monotone non-decreasing.
    \end{enumerate}
\end{conjecture}

There are some line segments of the graph in \Cref{fig:3d} that appear to be horizontal, a phenomenon called ``plateau'' in~\cite{Bruinetal}, following~\cite{Botella-Soler-et-al}. In \Cref{fig:horizontals} this is much clearer: each curve has a flat section toward the right. For each value of $b$, the flat section of the curve occurs for $a \in [-1,-\frac1{b+1}]$, so this is described explicitly in \Cref{horizontals}.
\begin{conjecture} \label{horizontals}
	If $b \le \half$ and $-1 \le a \le -\frac1{b+1}$ then $h_{\rm top}(f_{a,b}) = h_{\rm top}(f_{-1,b})$. That is, $h_{\rm top}(f_{a,b})$ is independent of $a$ in the region $\{ (a,b) \in \Params \,|\, b \le \half,\, -1 \le a \le -\frac1{b+1} \}$.
\end{conjecture}

One method to prove \Cref{horizontals} would be to prove that $\psi_{-1,b} = \psi_{-1/(b+1), b}$ for each $b \in [0,1]$. For some individual (rational) values of $b$, the authors have found a recoding from $(-1,b)$ to $(-\frac1{b+1},b)$, similar to the recoding from $(-1,1)$ to $(-\half,\half)$ presented in \Cref{sec:A and H}. However, it is not clear how to generalize those recodings to other $b \in [0,\half]$.

\begin{figure}[hbt]
	\definecolor{color1}{rgb}{1,0.25,0.75}
	\definecolor{color2}{rgb}{0,0,1}
	\definecolor{color3}{rgb}{0,0.75,0.75}
	\begin{tabular}{c@{\quad}c}
		\raisebox{0.3em}{\color{color1}\rule{0.67cm}{1.5pt}} $h_{\rm top}(f_{a,2/5})$
		\quad
		\raisebox{0.3em}{\color{color2}\rule{0.67cm}{1.5pt}} $h_{\rm top}(f_{a,1/3})$
		\quad
		\raisebox{0.3em}{\color{color3}\rule{0.67cm}{1.5pt}} $h_{\rm top}(f_{a,1/4})$
		\\
        \includegraphics[width=0.6\textwidth]{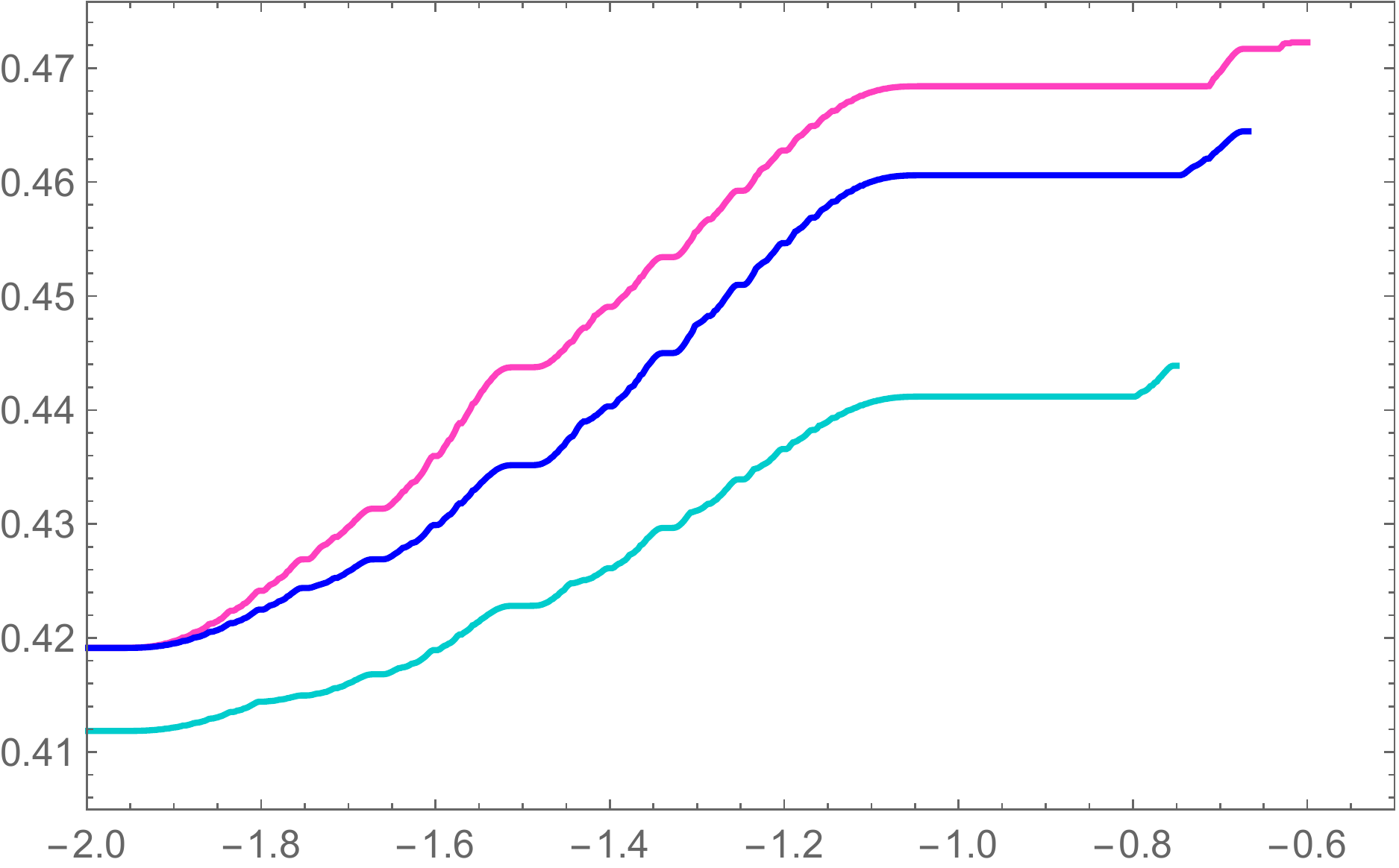}
		&
		\raisebox{1cm}{\begin{tikzpicture}[scale=3]
		    \def\Inf{1.25}
			\drawP[black!50]{black!15}{\Inf}
			\draw [thin] (-\Inf,0) node [left] {\tiny$a$\!} -- (0,0) -- (0,\Inf) node [above] {\tiny$b$};
			\draw [color1, ultra thick] (-\Inf,2/5) -- (-3/5,2/5);
			\draw [color2, ultra thick] (-\Inf,1/3) -- (-2/3,1/3);
			\draw [color3, ultra thick] (-\Inf,1/4) -- (-3/4,1/4);
			
			\draw (-1,0) node [below] {\tiny$-1$} (-1/2,0) node [below] {\tiny$-1/2$} (0,1/2) node [right] {\tiny$1/2$} (0,1) node [right] {\tiny$1$};
			%\draw (-1,1/2) -- (-1,0) [dashed,domain=0:1/2,variable=\b,samples=25] plot ({-1/(\b+1)},\b) -- (-1,1/2);
			\draw [dashed,domain=0:\Inf,variable=\b,samples=25] plot ({-1/(\b+1)},\b);
			\draw [dashed] (-1,0) -- (-1,1) -- (-1/2,1) -- (-1/2,1/2) -- (-1,1/2);
		\end{tikzpicture}}
		\\[-3pt]
		$a$
	\end{tabular}
	\caption{Plots of entropy for some fixed values of $b$}
	\label{fig:horizontals}
\end{figure}
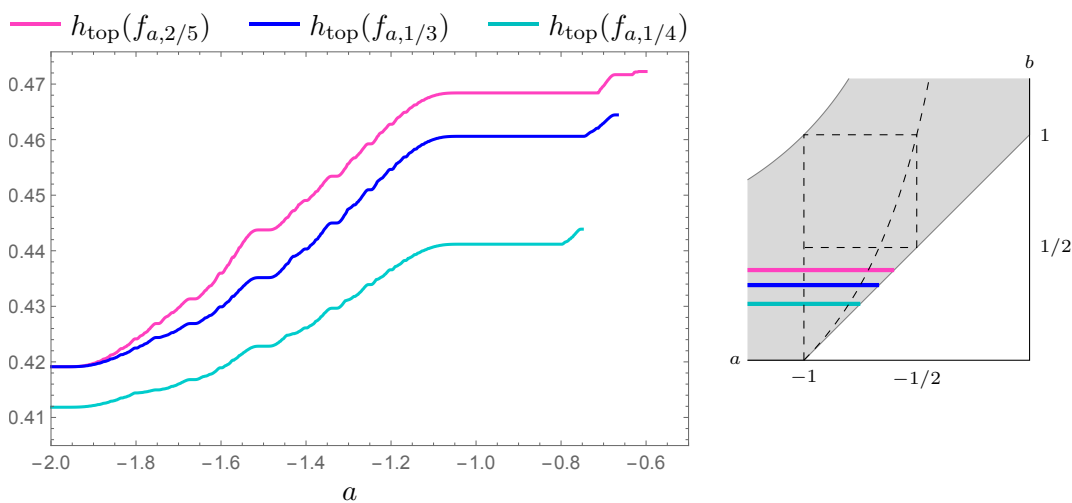

Although $f_{-1,1}$ (Artin) and $f_{-1,0}$ are well-studied, we do not have any explicit formula for $h_{\rm top}(f_{-1,b})$ in general; numerically calculated values are shown in \Cref{fig:outside of square}. This figure also shows entropy of $f_{b-1,b}$, the one-parameter family along the boundary of $\Params$ that is conceptually similar to the so-called Japanese continued fractions. Note that the two curves in \Cref{fig:outside of square} intersect \textit{only} at $(0,\log(\kappa))$ and $(\half,\log(\frac{1+\sqrt5}2))$; for all $0<b<\half$ we have (numerically) that $h_{\rm top}(f_{b-1,b}) > h_{\rm top}(f_{-1,b})$. Also, the fact that the right-half of the orange curve is flat is implied by \Cref{main}, and the symmetry of the purple curve is because $f_{a,b}$ and $f_{-b,-a}$ are conjugate.
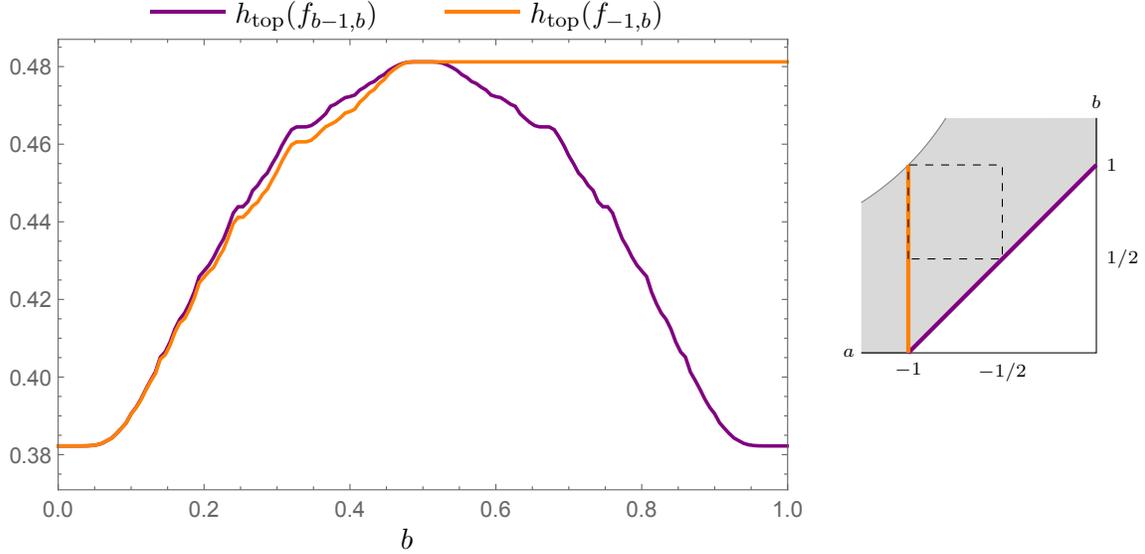
\begin{figure}[hbt]
	\begin{tabular}{c@{\quad}c}
		\raisebox{0.3em}{\color{red!50!blue}\rule{1cm}{1.5pt}} $h_{\rm top}(f_{b-1,b})$
		\qquad
		\raisebox{0.3em}{\color{orange}\rule{1cm}{1.5pt}} $h_{\rm top}(f_{-1,b})$
		\\[2pt]
		\includegraphics[width=0.7\textwidth]{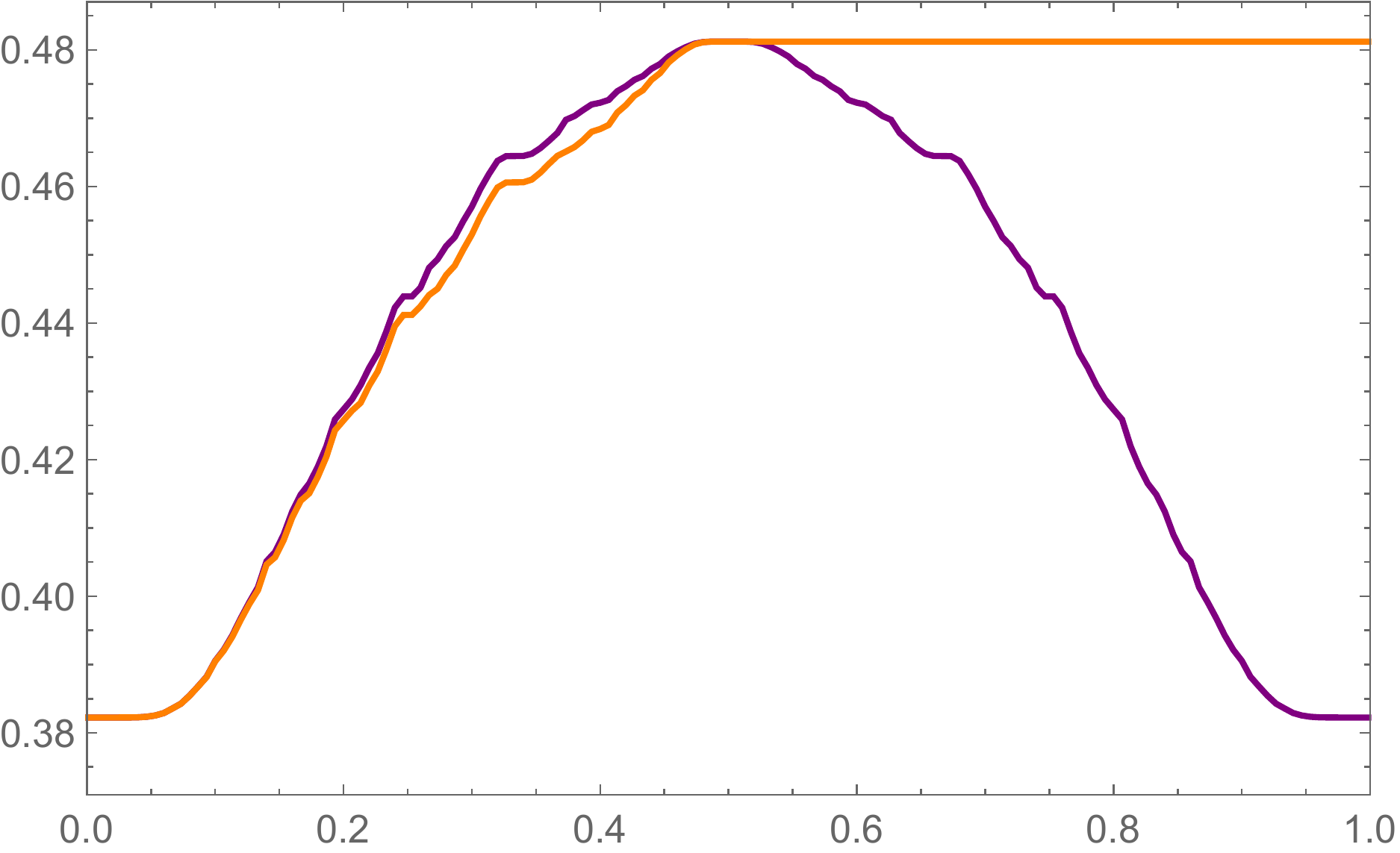}
		&
		\raisebox{1.75cm}{\begin{tikzpicture}[scale=2.5]
			\drawP[black!50]{black!15}{1.25}
			\draw [thin] (-1.25,0) node [left] {\tiny$a$\!} -- (0,0) -- (0,1.25) node [above] {\tiny$b$};
			\draw [red!50!blue, ultra thick] (-1,0) -- (0,1);
			\draw [orange, ultra thick] (-1,0) -- (-1,1);
			
			\draw (-1,0) node [below] {\tiny$-1$} (-1/2,0) node [below] {\tiny$-1/2$} (0,1/2) node [right] {\tiny$1/2$} (0,1) node [right] {\tiny$1$};
			\draw [dashed] (-1,1/2) rectangle (-1/2,1);
		\end{tikzpicture}}
		\\[-3pt]
		$b$
	\end{tabular}
	\caption{Plots of entropy for one-parameter families}
	\label{fig:outside of square}
\end{figure}

\subsection{Cycle property (matching) and entropy locking} \label{sec:matching}
In 2019, Bruin, Carminati, Marmi, and Profeti~\cite{Bruinetal} proved that entropy locking for a one-parameter family of affine maps of an interval with a single point of discontinuity occurs when the two orbits of the discontinuity point {\em match} after the same number of iterations (a property called {\em neutral matching}). Already in 2010 the authors in \cite{KU-JMD} proved that the {\em matching} for $(a,b)$-continued fractions occurs for essentially all $(a,b)\in \Params$ (they called it {\em the cycle property}): specifically,~$a$ has the cycle property if there exist nonnegative integers $m_a$ and $k_a$ such that
\[
f_{a,b}^{m_a}(Sa) =f_{a,b}^{k_a}(Ta),
\]
and, similarly, $b$ has the cycle property if there exist nonnegative integers $m_b$ and $k_b$ such that
\[
f_{a,b}^{m_b}(T^{-1}b) =f_{a,b}^{k_b}(Sb).
\]
In \cite{KU-Dedicata} they proved the cycle property for boundary maps associated to Fuchsian groups, another example of piecewise continuous piecewise monotone maps of the circle. In the $(a,b)$-case, the upper and lower cycles may be of arbitrary length while in the Fuchsian case the cycles are always the same length. In the Fuchsian case  we proved ``entropy rigidity'', that is, entropy locking on the entire parameter space~\cite{AKU-Rigidity}. In the $(a,b)$-case, we proved entropy locking in the golden square~$\Square$. In fact, generic $(a,b) \in \Square$ has neutral matching (i.e., $m_a=k_a$ and $m_b=k_b$), although the lengths can be arbitrary large (this can be proved by carefully following the analysis of~\cite[Sections~4 and~8]{KU-JMD}). % pages 649--650 and Section 8

\smallskip
It \textit{might} be possible to prove that $h_{\rm top}$ is constant in the ``golden square'' directly from the neutral matching property. Such a proof would be in the spirit of Bruin et al.~\cite{Bruinetal}. In the cocompact Fuchsian setting of~\cite{KU-Dedicata,AKU-Rigidity} the cycles are always the same length, so this argument, if possible, would provide alternative proofs of~\cite[Theorem 1]{AKU-Rigidity} and this paper's~\Cref{main} together.

\bibliographystyle{plain}
\bibliography{ab-references}

\begin{thebibliography}{10}

\bibitem{AKU-Rigidity}
A.~Abrams, S.~Katok, and I.~Ugarcovici.
\newblock Rigidity of topological entropy of boundary maps associated to
  {F}uchsian groups.
\newblock To appear in \textit{A Vision for Dynamics in the 21st Century},
  Cambridge University Press.

\bibitem{AF84}
R.~Adler and L.~Flatto.
\newblock The backward continued fraction map and geodesic flow.
\newblock {\em Ergodic Theory and Dynamical Systems}, 4(4):487–492, 1984.

\bibitem{AKM}
R.~Adler, A.~Konheim, and M.~McAndrew.
\newblock Topological entropy.
\newblock {\em Trans. Amer. Math. Soc.}, 114:309--319, 1965.

\bibitem{ALM}
L.~Alsed{\`a}, J.~Llibre, and M.~Misiurewicz.
\newblock {\em Combinatorial Dynamics and Entropy in Dimension One, Second
  Edition}.
\newblock Advanced Series in \textit{Nonlinear Dynamics 5}. World Scientific
  Publishing Co. Inc., River Edge, NJ, 2000.

\bibitem{AM}
L.~Alsed{\`a} and M.~Misiurewicz.
\newblock Semiconjugacy to a map of a constant slope.
\newblock {\em Discrete {\sl\&} Continuous Dynamical Systems B},
  20(10):3403--3413, 2015.

\bibitem{Botella-Soler-et-al}
V~Botella-Soler, J~A Oteo, J~Ros, and P~Glendinning.
\newblock Lyapunov exponent and topological entropy plateaus in piecewise
  linear maps.
\newblock {\em Journal of Physics A}, 46(12):26, 2013.

\bibitem{B7173}
R.~Bowen.
\newblock Entropy for group endomorphisms and homogeneous spaces.
\newblock {\em Trans. Amer. Math. Soc.}, 153:401--414, 1971.
\newblock Erratum: 181:509--510, 1973.

\bibitem{B79}
R.~Bowen.
\newblock Invariant measures for {M}arkov maps of the interval.
\newblock {\em Comm. Math. Phys.}, 69:1--17, 1979.

\bibitem{Bruinetal}
H.~Bruin, C.~Carminati, S.~Marmi, and A.~Profeti.
\newblock Matching in a family of piecewise affine maps.
\newblock {\em Nonlinearity}, 32(1):172--208, 2019.

\bibitem{CM}
D.~Cosper and M.~Misiurewicz.
\newblock Entropy locking.
\newblock {\em Fundamenta Mathematicae}, 241:83--96, 2018.

\bibitem{dMvS}
W.~de~Melo and S.~van Strien.
\newblock {\em One-dimensional dynamics}.
\newblock Ergebnisse der Mathematik und ihrer Grenzgebiete. 3. Folge. Springer,
  Berlin, Germany, 1993.

\bibitem{D70}
E.~Dinaburg.
\newblock The relation between topological entropy and metric entropy.
\newblock {\em Soviet Math. Dokl.}, 11:13--16, 1970.

\bibitem{KH}
A.~Katok and B.~Hasselblatt.
\newblock {\em Introduction to the Modern Theory of Dynamical Systems}.
\newblock Encyclopedia of Mathematics and its Applications. Cambridge
  University Press, 1995.

\bibitem{K96}
S.~Katok.
\newblock Coding of closed geodesics after {G}auss and {M}orse.
\newblock {\em Geometriae Dedicata}, 63(2), 1996.

\bibitem{KU-CWI}
S.~Katok and I.~Ugarcovici.
\newblock Arithmetic coding of geodesics on the modular surface via continued
  fractions.
\newblock {\em CWI Tracts}, 135:59--77, 2005.

\bibitem{KU-JMD}
S.~Katok and I.~Ugarcovici.
\newblock Structure of attractors for $(a,b)$-continued fraction
  transformations.
\newblock {\em Journal of Modern Dynamics}, 4:637--691, 2010.

\bibitem{KU-ETDS}
S.~Katok and I.~Ugarcovici.
\newblock Applications of $(a,b)$-continued fraction transformations.
\newblock {\em Ergodic Theory and Dyn. Systems}, 32:755--777, 2012.

\bibitem{KU-Dedicata}
S.~Katok and I.~Ugarcovici.
\newblock Structure of attractors for boundary maps associated to {F}uchsian
  groups.
\newblock {\em Geometriae Dedicata}, 191:171--198, 2017.

\bibitem{Krengel}
U.~Krengel.
\newblock Entropy of conservative transformations.
\newblock {\em Z. Wahrscheinlichkeitstheorie und Verw. Gebiete}, 7:161--181,
  1967.

\bibitem{MSz}
M.~Misiurewicz and W.~Szlenk.
\newblock Entropy of piecewise monotone mappings.
\newblock {\em Studia Mathematica}, 67:45--63, 1980.

\bibitem{MZ}
M.~Misiurewicz and K.~Ziemian.
\newblock {\em Horseshoes and entropy for piecewise continuous piecewise
  monotone maps}, pages 489--500.
\newblock World Sci. Publ., River Edge, NJ, 1992.

\bibitem{Parry64}
W.~Parry.
\newblock Intrinsic {M}arkov chains.
\newblock {\em Trans. Amer. Math. Soc.}, 112:55--55, 1964.

\bibitem{Parry66}
W.~Parry.
\newblock Symbolic dynamics and transformations of the unit interval.
\newblock {\em Trans. Amer. Math. Soc.}, 122:368--378, 1966.

\end{thebibliography}

\end{document}